\documentclass[12pt,twoside]{amsart}
\usepackage{amsmath, amsthm, amscd, amsfonts, amssymb, graphicx, color}
\usepackage[bookmarksnumbered, plainpages]{hyperref}

\newtheorem{thm}{Theorem}[section]

\newtheorem{lem}[thm]{Lemma}

\numberwithin{equation}{section}
\def\pn{\par\noindent}

\begin{document}

Advances in Geometry, 2016, Vol. 16, No. 3, 329--335

\vskip 1 true cm

\title{\bf Inequalities for Casorati curvatures of submanifolds in real space forms}
\author{Pan Zhang, Liang Zhang$^{\ast}$}

\thanks{{\scriptsize
\newline $^{\ast}$Corresponding author\\
\hskip -0.4 true cm \textit{2010 Mathematics Subject Classification.} 53C40, 49K35.
\newline This paper was supported by the Foundation for Excellent Young Talents of Higher
Education of Anhui Province (Grant NO.2011SQRL021ZD)
}}

\maketitle

\begin{abstract}
 By using T. Oprea's optimization methods on submanifolds, we give another proof of the inequalities relating the normalized $\delta-$Casorati curvature $\hat{\delta}_c(n-1)$ for submanifolds in real space forms. Also, inequalities relating the normalized $\delta-$Casorati curvature $\delta_C(n-1)$ for submanifolds in real space forms are obtained. Besides, we characterize a kind of Casorati ideal hypersurface of Euclidean 4-space. We also show that this kind of Casorati ideal hypersurface is rigid.
\end{abstract}

\vskip 0.2 true cm

%------------------------------------------------------------------------------------%

\pagestyle{myheadings}
\markboth{\rightline {\scriptsize P. Zhang, L. Zhang}}
         {\leftline{\scriptsize Inequalities for Casorati curvatures of submanifolds in real space forms}}

\bigskip
\bigskip

%------------------------------------------------------------------------------------%
%------------------------------------------------------------------------------------%

\section{ Introduction}
\vskip 0.4 true cm

In the theory of submanifolds, the following problem is fundamental: {\it to establish simple relationships between the main intrinsic invariants and the main
extrinsic invariants of the submanifolds} \cite{C1}. The basic relationships discovered until now are inequalities. Recently the study of this topic has attracted a lot of attention \cite{AM,CBYDF,BYH,CBYDFVL,BY4,LYL,FR,P3,PLW,UMW}.

On the other hand, it is well-known that the Casorati curvature of a submanifold in a Riemannian manifold is an extrinsic invariant defined as the normalized square of the length of the second fundamental form and it was preferred by Casorati over the traditional Gauss curvature because corresponds better with the common intuition of curvature \cite{CAS}. Later, Decu, Haesen and Verstraelen introduced the normalized $\delta-$Casorati curvatures $\delta_c(n-1)$ and $\hat{\delta}_c(n-1)$ and established inequalities involving $\delta_c(n-1)$ and $\hat{\delta}_c(n-1)$ for submanifolds in real space forms \cite{DS1}.

In \cite{DS1},  Decu {\it et al.} proved the following.

\begin{thm} (\cite{DS1}, Theorem 1) Let $M^n$ be an $n-$dimensional submanifold in an $(n+p)-$dimensional real space form $N^{n+p}(\widetilde{c})$ of constant sectional curvature $\widetilde{c}$, then we have
\begin{equation} \label{eq1.1}
\rho\leq \hat{\delta}_c(n-1)+\widetilde{c},
\end{equation}
where $\rho$ is the normalized scalar curvature of $M$. Moreover, the equality case holds if and only if $M^n$ is an invariantly quasi-umbilical submanifold with trivial normal connection in $N^{n+p}(\widetilde{c})$, such that with respect to suitable frames, the shape operators $A_r=A_{e_r}$, $r\in\{n+1,\cdots,n+p\}$, take the following forms:
$$A_{n+1}=\left(
  \begin{array}{cccccc}
   2\lambda  & 0 & 0 & \cdots & 0 & 0\\
   0  & 2\lambda & 0 & \cdots & 0 & 0\\
   0  & 0 & 2\lambda & \cdots & 0 & 0\\
    \vdots & \vdots &\vdots& \ddots &  \vdots &\vdots \\
    0&0&0&\cdots&2\lambda&0\\
    0&0 &0& \cdots &0&\lambda \\
  \end{array}
\right),\ \
A_{n+2}=\cdots=A_{n+p}=0.
$$
\end{thm}

The proof of Theorem 1.1 is based on an optimization procedure by showing that the quadratic polynomial in the components of the second fundamental form is parabolic. And the above approach was successfully applied to establish inequalities involving Casorati curvatures for different submanifolds in various ambient spaces  \cite{JGE,DS2,VG,VBGE}. In this paper, our main purpose is to present a new approach to establish inequalities for the Casorati curvatures. In Section 3, we give another proof of Theorem 1.1 by using T. Oprea's optimization methods on Riemannian submanifolds \cite{O1}. In Sections 4 and 5, we not only establish an inequality in terms of the normalized $\delta-$Casorati curvature $\delta_C(n-1)$ for submanifolds in real space forms but also characterize several kinds of Casorati ideal submanifolds ("ideal" in the sense of Bang Yen Chen).

%------------------------------------------------------------------------------------%

\section{ Preliminaries}
\vskip 0.4 true cm

Let $(M^n,g)$ be an $n-$dimensional submanifold in an $(n+p)-$dimensional real space form $(N^{n+p}(\widetilde{c}),\overline{g})$ of constant sectional curvature $\widetilde{c}$. The Levi-Civita connections on $N$ and $M$ will be denoted by $\overline{\nabla}$ and $\nabla$, respectively. For all $X,Y\in C^{\infty}(TM),\ \  Z\in C^{\infty}(T^{\bot}M)$, the Gauss and Weingarten formulas can be expressed by
$$\overline{\nabla}_XY=\nabla_XY+h(X,Y), \ \ \ \overline{\nabla}_XZ=-A_ZX+\nabla_X^{\bot}Z, $$
where $h$ is the second fundamental form of $M$, $\overline{\nabla}$ is the normal connection and the shape operator $A_Z$ of $M$ is given by
$$g(A_ZX,Y)=\overline{g}(h(X,Y),Z).$$

We denote by $\overline{R}$ and $R$ the curvature tensors associated to $\overline{\nabla}$ and $\nabla$, then the Gauss equation is given by
\begin{equation}\label{eq2.1}
R(X,Y,Z,W)=\overline{R}(X,Y,Z,W)+g(h(X,Z),h(Y,W))-g(h(X,W),h(Y,Z)).
\end{equation}

In $N^{n+p}$ we choose a local orthonormal frame $e_1,\cdots,e_n,e_{n+1},\cdots,e_{n+p}$, such that, restricting to $M^n$, $e_1,\cdots,e_n$ are tangent to $M^n$.  We write $h_{ij}^{r}=g(h(e_i,e_j),e_{r})$. Then the \textit{mean curvature vector} $H$ is given by $$H=\sum\limits_{r=n+1}^{n+p}(\frac{1}{n}\sum\limits_{i=1}^nh_{ii}^{r})e_{r}$$
and the squared norm of $h$ over dimension $n$ is denoted by  $\mathcal{C}$ and is called the Casorati curvature of the submanifold $M$. Therefore we have
$$\mathcal{C}=\frac{1}{n}\sum_{r=n+1}^{n+p}\sum_{i,j=1}^n(h_{ij}^r)^2.$$

The submanifold $M^n$ is called to be \textit{totally geodesic} if $h=0$ and {\it minimal} if $H=0$. Besides, $M^n$ is called \textit{invariantly quasi-umbilical} if there exist $p$ mutually orthogonal unit normal vectors $\xi_{n+1}, \cdots, \xi_{n+p}$ such that the shape operators with respect to all
directions $\xi_{r}$ have an eigenvalue of multiplicity $n-1$ and that for each $\xi_{r}$ the distinguished
eigendirection is the same \cite{JGE}.

Let $K(e_i\wedge e_j),\ 1\leq i<j\leq n$, denote the \textit{sectional curvature} of the plane section spanned by $e_i$ and $e_j$. Then the \textit{scalar curvature} of $M^n$ is given by
$$\tau=\sum_{i<j}K(e_i\wedge e_j),$$
and the normalized scalar curvature $\rho$ is defined by
$$\rho=\frac{2\tau}{n(n-1)}.$$

Suppose $L$ is an $l$-dimensional subspace of $T_xM$,\ $x\in M$, \ $l\geq2$ and $\{e_1,\cdots,e_l\}$ an orthonormal basis of $L$. Then the \textit{scalar curvature $\tau(L)$} of the $l-$plane $L$ is given by
$$
\tau(L)=\sum_{1\leq \mu<\nu\leq l}K(e_{\mu}\wedge e_{\nu}),
$$
and the Casorati curvature $\mathcal{C}(L)$ of the subspace $L$ is defined as
$$\mathcal{C}(L)=\frac{1}{r}\sum_{r=n+1}^{n+p}\sum_{i,j=1}^l(h_{ij}^r)^2.$$

Following \cite{DS1,VG}, we can define {\it the normalized $\delta-$Casorati curvature $\delta_c(n-1)$ and $\hat{\delta}_c(n-1)$} by
\begin{equation}\label{eq2.2}
[\delta_c(n-1)]_x=\frac{1}{2}\mathcal{C}_x+\frac{n+1}{2n(n-1)}\textmd{inf}\big\{\mathcal{C}(L)\mid L \ \ \textmd{a hyperplane of} \ \ T_xM \big\},
\end{equation}
and
$$[\hat{\delta}_c(n-1)]_x=2\mathcal{C}_x-\frac{2n-1}{2n}\textmd{sup}\big\{\mathcal{C}(L)\mid L \ \ \textmd{a hyperplane of}\ \ T_xM \big\}.$$

\vskip 0.25 true cm
 \vskip 0mm \noindent{{\bf Remark 2.1}}\ \ {\it It should be noted that the normalized $\delta-$Casorati curvatures vanish trivially for $n=2$.}
\vskip 0.25 true cm

For later use, we provide a brief review of T. Oprea's optimization methods on submanifolds from \cite{O1}.

Let $(N_2, \overline{g})$ be a Riemannian manifold, $N_1$ be a Riemannian
submanifold of it, $g$ be the metric induced on $N_1$ by $\overline{g}$ and $f:N_2\longrightarrow \mathbb{R}$ be a
differentiable function.

Following \cite{O1} we considered the constrained extremum problem
\begin{equation}\label{eq2.3}
\min_{x\in N_1}f(x),
\end{equation}
then we have

\begin{lem} (\cite{O1}) If $x_0\in N_1$ is the solution of the problem \eqref{eq2.3}, then

  $i)(gradf)(x_0)\in T_{x_{0}}^\perp N_1;$

 $ii)$ the bilinear form
$$\mathcal{A}:T_{x_{0}}N_1\times T_{x_{0}}N_1\rightarrow \mathbb{R},
$$
$$\mathcal{A}(X, Y)=Hess_f(X, Y)+\overline{g}(h(X, Y), (gradf)(x_0)),
$$
is positive semidefinite, where $h$ is the second fundamental form of $N_1$ in $N_2$ and $gradf$ is the gradient of function $f$.
\end{lem}

In \cite{O2}, the above lemma was successfully applied to improve an inequality relating $\delta(2)$ obtained in \cite{C2}. Later, Chen extended the improved inequality to the general inequalities involving $\delta(n_1,\cdots,n_k)$ \cite{CBYDF}. More details of $\delta-$invariants $\delta(n_1,\cdots,n_k)$ can be found in \cite{BY5,PLW}.

%------------------------------------------------------------------------------------%

\section{ Another proof of Theorem 1.1}
\vskip 0.4 true cm

From \eqref{eq2.1} we have
$$R_{ijij}=\widetilde{c}+\sum_{r=n+1}^{n+p}[h_{ii}^{r}h_{jj}^{r}-(h_{ij}^{r})^2],$$
which implies
\begin{equation} \label{eq3.1}
2\tau=n^2\parallel H\parallel^2-n\mathcal{C}+n(n-1)\widetilde{c}.
\end{equation}
Consider the following function $\mathcal{P}$ which is a quadratic polynomial in the components of the second fundamental form:
\begin{equation} \label{eq3.2}
\mathcal{P}=2n(n-1)\mathcal{C}+\frac{(n-1)(1-2n)}{2}\mathcal{C}(L)-2\tau+n(n-1)\widetilde{c}.
\end{equation}
Assuming, without loss of generality, that $L$ is spanned by $e_1,e_2,\cdots,e_{n-1}$, combining \eqref{eq3.1} it follows that
\begin{align}
\mathcal{P}&=\sum_{r=n+1}^{n+p}\{\frac{2n-3}{2}\sum_{i=1}^{n-1}(h_{ii}^r)^2+2(n-1)(h_{nn}^r)^2+(2n-1)\sum_{1\leq i<j\leq n-1}(h_{ij}^r)^2\notag\\
&~~~~~~~~+2(2n-1)\sum_{i=1}^{n-1}(h_{in}^r)^2-2\sum_{1\leq i<j\leq n}h_{ii}^rh_{jj}^r\}\notag\\
&\geq \sum_{r=n+1}^{n+p}\{\frac{2n-3}{2}\sum_{i=1}^{n-1}(h_{ii}^r)^2+2(n-1)(h_{nn}^r)^2-2\sum_{1\leq i<j\leq n}h_{ii}^rh_{jj}^r\}.\notag
\end{align}

For $r=n+1,\cdots,n+p$, let us consider the quadratic forms
$$f_{r}: \mathbb{R}^n\longrightarrow \mathbb{R},$$
$$f_{r}(h_{11}^{r},\cdots,h_{nn}^{r})=\frac{2n-3}{2}\sum_{i=1}^{n-1}(h_{ii}^r)^2+2(n-1)(h_{nn}^r)^2-2\sum_{1\leq i<j\leq n}h_{ii}^rh_{jj}^r$$
and the constrained extremum problem
$$\min f_{r},$$
$$\textmd{subject}\ \ \textmd{to}\ \ \digamma: h_{11}^{r}+\cdots+h_{nn}^{r}=k^{r},$$
where $k^{r}$ are real constants. In fact, for a fixed normal vector $e_r$, $k^r$ is the trace of the matrix $(h^r_{ij})$, which is invariant no matter how the entries $h^r_{ij}$ change.

The  partial derivatives of the function $f_{r}$ are
\begin{equation} \label{eq3.3}
\frac{\partial f_{r}}{\partial h_{11}^{r}}=(2n-3)h_{11}^r-2\sum_{i=2}^nh_{ii}^r
\end{equation}
\begin{equation} \label{eq3.4}
\frac{\partial f_{r}}{\partial h_{22}^{r}}=(2n-3)h_{22}^r-2h_{11}^r-2\sum_{i=3}^nh_{ii}^r
\end{equation}
$$\cdots\cdots$$
\begin{equation} \label{eq3.5}
\frac{\partial f_{r}}{\partial h_{n-1,n-1}^{r}}=(2n-3)h_{n-1,n-1}^r-2h_{nn}^r-2\sum_{i=1}^{n-2}h_{ii}^r
\end{equation}
\begin{equation} \label{eq3.6}
\frac{\partial f_{r}}{\partial h_{nn}^{r}}=4(n-1)h_{nn}^r-2\sum_{i=1}^{n-1}h_{ii}^r.
\end{equation}

For an optimal solution $(h_{11}^{r},h_{22}^{r},\cdots,h_{nn}^{r})$ of the problem in question, the vector $\textmd{grad}f_{r}$ is normal at $\digamma$ that is, it is colinear with the vector $(1, 1, ..., 1)$. From \eqref{eq3.3}, \eqref{eq3.4}, \eqref{eq3.5} and \eqref{eq3.6}, it follows that a critical point of the considered problem has the form
\begin{equation} \label{eq3.7}
(h_{11}^{r},h_{22}^{r},\cdots,h_{n-1,n-1}^{r},h_{nn}^{r})=(2t^{r},2t^{r},\cdots,2t^{r},t^{r}).
\end{equation}
As $\sum\limits_{i=1}^nh_{ii}^{r}=k^{r}$, by using \eqref{eq3.7}, we have
\begin{equation} \label{eq3.8}
h_{11}^{r}=h_{22}^{r}=\cdots=h_{n-1,n-1}^{r}=\frac{2}{2n-1}k^{r},\ \  h_{nn}^{r}=\frac{1}{2n-1}k^{r}.
\end{equation}

We fix an arbitrary point $x \in \digamma$. The bilinear form $\mathcal{A}: T_x\digamma \times T_x\digamma\longrightarrow  \mathbb{R}$ has the expression
$$\mathcal{A}(X,Y) = \textmd{Hess}f_{r}(X,Y)+\langle h'(X,Y),(\textmd{grad}f_{r})(x)\rangle,$$
where $h'$ is the second fundamental form of $\digamma$ in $\mathbb{R}^n$ and $\langle , \rangle$ is the standard inner-product on $\mathbb{R}^n$.
In the standard frame of $\mathbb{R}^n$, the Hessian of $f_{r}$ has the matrix
$$
\left(
  \begin{array}{cccccc}
   2n-3  & -2 & -2 & \cdots & -2 & -2\\
   -2 & 2n-3 & -2 & \cdots & -2 & -2\\
   -2  & -2 & 2n-3 & \cdots & -2 & -2\\
    \vdots & \vdots &\vdots& \ddots &  \vdots &\vdots \\
    -2&-2&-2&\cdots&2n-3&-2\\
    -2&-2 &-2& \cdots &-2&4(n-1) \\
  \end{array}
\right).
$$
As $\digamma$ is totally geodesic in $\mathbb{R}^n$ , considering a vector $X$ is tangent to $\digamma$ at the arbitrary point $x$ on $\digamma$, that is, verifying the relation
$\sum\limits_{i=1}^nX_i=0$, we have
\begin{align}
\mathcal{A}(X,X)&=(2n-1)\sum_{i=1}^{n-1}X_i^2+(4n-2)X_n^2-2(X_1+X_2+\cdots+X_n)^2\notag\\
&=(2n-1)\sum_{i=1}^{n-1}X_i^2+(4n-2)X_n^2\notag\\
&\geq 0.\notag
\end{align}
Consequently the point $(h_{11}^{r},h_{22}^{r},\cdots,h_{nn}^{r})$ given by \eqref{eq3.8} is a global minimum point, here we used Lemma 2.1.
Inserting \eqref{eq3.8} into $f_r$ we have $f_r=0$. Hence we have
\begin{equation} \label{eq3.9}
\mathcal{P}\geq 0.
\end{equation}

From \eqref{eq3.2} and \eqref{eq3.9} we can derive inequality \eqref{eq1.1}. The equality case of \eqref{eq1.1} holds if and only if we have the equality
in all the previous inequalities. Thus the shaper operators take the desired forms. Besides, $h_{ij}^r=0, \ \ i\neq j,\ \forall i,j,r$ means that the normal connection $\nabla^{\perp}$ is flat, or still, that the normal curvature tensor $R^{\perp}$, i.e., the curvature tensor of the normal connection is trivial.

%------------------------------------------------------------------------------------%

\section{ Inequalities for the modified normalized $\delta-$Casorati curvature}
\vskip 0.4 true cm

It was pointed that the coefficient $\frac{n+1}{2n(n-1)}$ in \eqref{eq2.2} is inappropriate and must be replaced by $\frac{n+1}{2n}$ \cite{JGE}.  Because the normalized $\delta-$Casorati curvature $\delta_C(n-1)$ should be able to be recovered from the generalized normalized $\delta-$Casorati curvature, it would be more appropriate to define
$\delta_C(n-1)$ using the coefficient $\frac{n+1}{2n}$. More details of generalized normalized $\delta-$Casorati curvature can be found in \cite{JGE,DS2}. Following \cite{JGE}, we define the normalized $\delta-$Casorati curvature $\delta_C(n-1)$ by
$$[\delta_C(n-1)]_x=\frac{1}{2}\mathcal{C}_x+\frac{n+1}{2n}\textmd{inf}\big\{\mathcal{C}(L)\mid L \ \ \textmd{a hyperplane of} \ \ T_xM \big\}.$$

Using T. Oprea's optimization methods on Riemannian submanifolds and considering the following quadratic polynomial in the components of the second fundamental form:
$$\mathcal{Q}=\frac{1}{2}n(n-1)\mathcal{C}+\frac{1}{2}(n-1)(n+1)\mathcal{C}(L)-2\tau+n(n-1)\widetilde{c},$$
 we establish the following inequalities in terms of $\delta_C(n-1)$ for submanifolds of a real space form.

\begin{thm} Let $M^n$ be an $n-$dimensional submanifold in an $(n+p)-$dimensional real space form $N^{n+p}(\widetilde{c})$ of constant sectional curvature $\widetilde{c}$, then we have
\begin{equation} \label{eq4.1}
\rho\leq \delta_C(n-1)+\widetilde{c},
\end{equation}
where $\rho$ is the normalized scalar curvature of $M$. Moreover, the equality case holds if and only if $M^n$ is an invariantly quasi-umbilical submanifold with trivial normal connection in $N^{n+p}(\widetilde{c})$, such that with respect to suitable frames, the shape operators $A_r=A_{e_r}$, $r\in\{n+1,\cdots,n+p\}$, take the following forms:
$$A_{n+1}=\left(
  \begin{array}{cccccc}
   \lambda  & 0 & 0 & \cdots & 0 & 0\\
   0  & \lambda & 0 & \cdots & 0 & 0\\
   0  & 0 & \lambda & \cdots & 0 & 0\\
    \vdots & \vdots &\vdots& \ddots &  \vdots &\vdots \\
    0&0&0&\cdots&\lambda&0\\
    0&0 &0& \cdots &0&2\lambda \\
  \end{array}
\right),\ \
A_{n+2}=\cdots=A_{n+p}=0.
$$
\end{thm}

\vskip 0.4 true cm

\vskip 0mm \noindent{{\bf Remark 4.1}}\ \ {\it We omit the proof Theorem 4.1 since it is essentially similar to that of Theorem 1.1.}

%------------------------------------------------------------------------------------%

\section{Casorati ideal submanifolds in real space forms}
\vskip 0.4 true cm

The notion of ideal immersions was introduced by Chen in the 1990s. Roughly speaking, an ideal immersion of a Riemannian manifold into a real space form is a nice isometric immersion which produces the least possible amount of tension from the ambient space at each point. Chen established many inequalities in terms of $\delta-$invariants and claimed that the submanifold satisfying the equality case is called ideal submanifold. Such submanifolds are also called Chen's submanifolds \cite{BY5}. The ideal submanifolds in real space forms and complex space forms have been characterized by Chen \cite{C2,BYH,CBYDFVL,BY4}. Besides, Einstein, conformally flat, semisymmetric, and Ricci-semisymmetric submanifolds satisfying Chen's inequality in real space forms were studied by Dillen, Petrovic, Verstraelen, \"{O}zg\"{u}r and Tripathi \cite{OCTMM,FML}. More details about ideal submanifolds, we refer to see \cite{BY5}.

Submanifolds for which the equality case of inequalities for the Casorati curvatures, will be called Casorati ideal submanifolds \cite{DS1,DS2}. In \cite{DS1}, the authors proved that

\begin{thm} (\cite{DS1},Corollary 4) The Casorati ideal submanifold with $n\geq 4$ for \eqref{eq1.1} is conformally flat submanifold with trivial normal connection.
\end{thm}

\begin{thm} (\cite{DS1},Corollary 5) The Casorati ideal submanifold for \eqref{eq1.1} is pseudo-symmetric manifold.
\end{thm}

In this paper, we have
\begin{thm} The Casorati ideal submanifold for \eqref{eq1.1} and \eqref{eq4.1} are Einstein if and only if they are totally geodesic submanifolds.
\end{thm}

\begin{proof} Just take Casorati ideal submanifold for \eqref{eq1.1} as an example. The equation of Gauss gives
$$
Ric(e_i)=(n-1)\widetilde{c}+2(2n-3)a^2,\ \ i=1,2,\cdots,n-1,
$$
$$
Ric(e_n)=(n-1)\widetilde{c}+2(n-1)a^2.
$$
As $M^n$ is Einstein, $Ric(e_i)=Ric(e_n), \ i=1,2,\cdots,n-1$. Thus $a=0$, which implies $h=0$.
\end{proof}

\vskip 0mm \noindent{{\bf Theorem 5.4.}}\ \ \textit{The Casorati ideal hypersurface $M^3$ for \eqref{eq4.1} in Euclidean 4-space $\mathbb{E}^4$  is congruent to}
$$(\frac{1}{a}\textmd{sd}(at,\frac{1}{\sqrt{2}})\textmd{sin}u,\frac{1}{a}\textmd{sd}(at,\frac{1}{\sqrt{2}})\textmd{cos}u\textmd{sin}v,
\frac{1}{a}\textmd{sd}(at,\frac{1}{\sqrt{2}})\textmd{cos} u \textmd{cos} v,
\frac{1}{2}\int_{0}^t \textmd{sd}^2(at,\frac{1}{\sqrt{2}})dt)$$
{\it for some positive real number $a$, where} $\textmd{sd}(,)$ {\it denotes Jacobi's elliptic function.}

\vskip 0.2 true cm

\vskip 0mm \noindent{{\bf Remark 5.1.}}\ \ \ {\it More details of Jacobi's elliptic functions can be found in} \cite{BY4,FB}.

\vskip 0.2 true cm

\begin{proof}
Theorem 4.1 implies that there exits an orthonormal frame $\{e_1,e_2,e_3,e_4\}$ such that the shape operator of $M^3$ with respect to this frame takes the following simple form:
\begin{equation} \label{eq5.1}
A=\left(
  \begin{array}{cccccc}
   \lambda  & 0 & 0 \\
   0  & \lambda & 0 \\
   0  & 0 & 2\lambda  \\
  \end{array}
\right).
\end{equation}

Clearly, \eqref{eq5.1} is a special case of (4.4) in \cite{BY4}. Then from Theorem 4.1 in \cite{BY4}, we derive the conclusion.
\end{proof}

Recall that an isometric immersion of a Riemannian $n-$manifold into a Euclidean $m-$space is called rigid if the isometric immersion is unique up to isometries of $\mathbb{E}^m$. From Theorem 4.2 in \cite{BY4} and Theorem 5.4 we have

\begin{thm}
The Casorati ideal hypersurface $M^3$ for \eqref{eq4.1} in Euclidean 4-space $\mathbb{E}^4$ is rigid.
\end{thm}

\vskip 0.4 true cm

{\bf Acknowledgements} We would like to thank Professors Teodor Oprea, Jaewon Lee and Mukut Mani Tripathi for the disscussion held on this topic. More over, we are thankful to the referees for their valuable comments and suggestions which improved the paper.
\vskip 0.4 true cm

\vskip 0.5 true cm

%-----------------------------------------------------------------------------
%-----------------------------------------------------------------------------

\bigskip
\bigskip

{\footnotesize \pn{\it Pan Zhang, Liang Zhang}\; \\
{School of Mathematics and Computer Science}\\
{Anhui Normal University}\\
{Anhui 241000}\\
{P. R. China}\\
{Email: 656257701@qq.com(P. Zhang); zhliang43@163.com(L. Zhang)}

\end{document}